\documentclass[a4paper,11pt,reqno]{amsart}
\usepackage{graphicx}
\usepackage{color}
\definecolor{rblue}{RGB}{39,64,139}
\usepackage{hyperref}
\hypersetup{pdftex,naturalnames=false,%
	draft=false,
	bookmarks=true,bookmarksopenlevel=1,%
	allcolors=rblue, %% color of normal internal links
	allbordercolors={.15 .25 .55},
	colorlinks=true,%% links are colored (falsifying will produce borders)
	}
\usepackage{amssymb}
\usepackage[initials]{amsrefs}
\usepackage[cmtip]{xypic}

\title[Harish-Chandra's volume formula]{Harish-Chandra's volume formula via \\
Weyl's Law and Euler-Maclaurin formula}
\author{Seunghun Hong}
\address{Mathematisches Institut, Busenstra{\ss}e 3--5, D-37073 G\"ottingen, Germany}
\email{shong@uni-goettingen.de}
\urladdr{diracoperat.org}

\theoremstyle{plain}
\newtheorem{thm}{Theorem}[section]
\newtheorem{lem}[thm]{Lemma}
\newtheorem{prop}[thm]{Proposition}

\theoremstyle{definition}

\newtheorem{notation}[thm]{Notation}
\theoremstyle{remark}
\newtheorem*{rmk}{Remark}

\newcommand{\C}{\mathbb{C}}

\newcommand{\Z}{\mathbb{Z}}
\newcommand{\N}{\mathbb{N}}
\newcommand{\grg}{\mathfrak{g}}
\newcommand{\grt}{\mathfrak{t}}
\DeclareMathOperator{\tr}{tr}
\DeclareMathOperator{\Str}{Str}
\newcommand{\vol}{\mathrm{vol}}

\DeclareMathOperator{\Ind}{Ind}
\DeclareMathOperator{\Ad}{Ad}

\DeclareMathOperator{\Td}{Td}

\begin{document}
\begin{abstract}
Harish-Chandra's volume formula shows that the volume of a flag manifold $G/T$, where the measure is induced by an invariant inner product on the Lie algebra of $G$, is determined up to a scalar by the algebraic properties of $G$. This article explains  how to deduce Harish-Chandra's formula from Weyl's law by utilizing the Euler-Maclaurin formula. This approach leads to a mystery that lies under the Atiyah-Singer index theorem.
\end{abstract}
\maketitle

\section{Introduction}

Harish-Chandra's volume formula \cite{harishchandra1}*{Lem.~4, p.~203} calculates the volume of a flag manifold $G/T$ by algebraic means. To wit, suppose the Lie algebra $\mathfrak{g}$ of $G$ is equipped with an $\Ad(G$)-invariant inner product, where $\Ad$ denotes the adjoint representation of $G$ on $\mathfrak{g}$. The inner product determines a unique invariant metric on $G$. Endow the quotient measure\footnote{%
	Let $p\colon G\to G/T$ be the canonical projection.
	If $\mu_G$ is a measure on $G$, then
	the quotient measure $\mu_{G/T}$ on $G/T$ is 
	defined by $\mu_{G/T}(U) = \mu_G(p^{-1}(U))$
	for every open subset $U$ of $G/T$.%
	} %
on $G/T$. Then 
\[ 
 \vol(G/T) =\prod_{\alpha\in\Phi^+}\frac{2\pi}{\langle\alpha,\rho\rangle}.
\]
Here $\Phi^+$ is the set of (selected) positive roots of $G$; $\rho$ is half the sum of the positive roots; and $\langle \cdot , \cdot \rangle$ is the inner product on the dual space $\grt^*$ of the Lie algebra of $T$ induced by the inner product on $\grg$. Since an invariant inner product on $\mathfrak{g}$ is unique up to a scalar factor, the formula shows that the volume is essentially determined by the algebraic properties of $G$. Apart from this, a significance of the volume of $G/T$ can be found in its appearance in the \mbox{Weyl} integration formula \cite{weyl1}*{$\S$~6}.

Proofs for the volume formula other than Harish-Chandra's own can be found in \mbox{Berline}, \mbox{Getzler} and \mbox{Vergne}~\cite{bgv}*{Cor.~7.27, p.~230}, \mbox{Duistermaat} and \mbox{Kolk}~\cite{duistermaatkolk}*{Eq.~3.14.13, p.~192}, \mbox{Fegan}~\cite{fegan}*{Thm.~1.4, p.~591}, Flensted-Jensen~\cite{flensted-jensen}*{Eq.~3.9, p.116}, and \mbox{Macdonald}~\cite{macdonald}*{p.~95}. The ones that are close to our approach would be that of \mbox{Duistermaat} and \mbox{Kolk}, and \mbox{Fegan}, as they explicitly depend on the property of the heat trace. But \mbox{Duistermaat} and \mbox{Kolk} exploits the \mbox{Weyl} integration formula, while \mbox{Fegan} relies on the \mbox{Poisson} summation formula. Our approach differs from them in that we utilize the Euler-Maclaurin formula; as the Euler-Maclaurin formula relates sums with integrals, the \mbox{Weyl} integration formula also plays a natural role.

Here is the outline of our proof. The starting point is what is known as \emph{Weyl's law} for closed manifolds, namely that, if $\Delta$ is the Laplacian of a closed Riemannian manifold $M$ of dimension $n$, then its \emph{heat trace}
\[ Z(t):= \tr(e^{t\Delta}), \]
for $t>0$, satisfies the asymptotic equality:
\begin{equation}
(4\pi t)^{n/2}Z(t) = \vol(M) + O(t),\label{eq:weylslaw}
\end{equation}
for $t\to 0+$ (\mbox{Minakshisundaram} and Pleijel~\cite{minak}). Now, as a result of Peter-Weyl theorem and Schur's lemma, the spectrum of the Laplacian is parametrized by the unitary dual $\hat G$ of $G$, and we have:
\[ Z(t) =\sum_{u\in\hat G} \dim(u)^2 e^{tC_u},\]
where $C_u$ is a constant that depends on  $u$. According to the representation theory of compact Lie groups, there is a one-to-one correspondence between $\hat G$ and $\Lambda\cap K$, where $\Lambda$ and $K$, respectively, are a certain lattice and cone in $\mathfrak{t}^*$. Hence,
\[ Z(t) =\sum_{\lambda\in\Lambda\cap K} d(\lambda)^2 e^{t\Omega_\lambda},\]
where $d(\lambda)$ is some function on $\mathfrak{t}^*$, and $\Omega_\lambda$ is a constant that depends on $\lambda$.
Applying the Euler-Maclaurin formula, we obtain the asymptotic equality:
\begin{equation}
 t^{n/2}Z(t)= t^{n/2}I(t) +O(t) \label{eq:httreumc}
\end{equation}
for $t\to0+$, where
\[ I(t)= \frac{\vol(T)}{(2\pi
)^{\dim(T)}}\int_{K} d(\lambda)^2e^{-t\|\lambda\|^2}\,d\lambda. \]
Owing to the invariance of $d(\lambda)^2$ and $\|\lambda\|$ relative to the Weyl group action, the domain of the above integral can be extended to whole $\mathfrak{t}^*$. Then, with the aid of Weyl integration formula, we arrive at:
\[ I(t) =  \frac{\vol(T)^2}{(2\pi)^{\dim(T)}\vol(G)}\Bigl(\prod_{\alpha\in\Phi^+}\frac{1}{\langle\alpha,\rho\rangle^2}\Bigr)
	\int_{\grg}  e^{-t\|X\|^2} \,dX. 
\]
The last Gaussian integral is easily evaluated. Comparing the two asymptotic equations~\eqref{eq:weylslaw} and~\eqref{eq:httreumc} then yields Harish-Chandra's formula. 

The first significance of this line of reasoning would be that once we are equipped with few key results in Lie theory and differential geometry, namely, the representation theory of compact Lie groups and Weyl's law, we can attain Harish-Chandra's formula by an elementary means of the Euler-Maclaurin formula. A greater significance might still lie ahead in the context of index theory. We briefly indicate this in the concluding remarks.

\section{Acknowledgement}
This work was done as a partial fulfillment of the requirements for the author's doctoral degree. He wishes to express his heartfelt gratitude to his  advisor Nigel Higson.

\section{Preliminary Remarks} \label{par:cptlgpstrth1} 
Throughout this article, $G$ is a compact connected Lie group, and $T$ is a maximal torus of $G$. The measure on $G/T$ under consideration is the quotient measure coming from an invariant measure on $G$. An invariant measure on $G$ is unique up to a scalar factor; moreover, it is induced by a bi-invariant metric on $G$, which is in turn induced by an $\Ad(G)$-invariant inner product on $\mathfrak{g}$. Henceforth we fix such an inner product and denote it by $\langle \cdot , \cdot \rangle$, and endow $G$ and $T$ the invariant metrics induced by it. Then we have
\[ \vol(G/T)=\frac{\vol(G)}{\vol(T)}.\]

We point out that, as far as $\vol(G/T)$ is concerned, we may further assume, without loss of generality, that the compact connected Lie group $G$ is semisimple and simply connected, for the following reasons. By the general theory of compact Lie groups, every compact connected Lie group $G$ satisfies an isomorphism
\begin{equation}
 G\cong (R\times S)/F, \label{eq:strcthmcclieg}
\end{equation}
where $R$ is a torus, $S$ is a compact, connected, simply connected, semisimple Lie group, and $F$ is a finite abelian subgroup of $R\times S$ (see Knapp~\cite{knapp}*{Thm.~4.29, p.~250}). Thus $F$ is contained in a maximal torus $\widetilde T$ of $R\times S$. Then $\widetilde T/F$ is a maximal torus of $(R\times S)/F$; let $T$ be the corresponding maximal torus of $G$ under the isomorphism~\eqref{eq:strcthmcclieg}. The Lie groups $G$, $R\times S$, and $(R\times S)/F$ all have isomorphic Lie algebras.  Hence the $\Ad(G)$-invariant inner product on $\grg$ induces a bi-invariant metric on the Lie groups $G$, $R\times S$, $(R\times S)/F$, and their respective maximal tori. Their volumes satisfy:
\[
\frac{\vol(R\times S)}{\vol(\widetilde T)}=\frac{\vol((R\times S)/F)}{\vol(\widetilde T/F)}=\frac{\vol(G)}{\vol(T)}.
\]
Therefore, for our purpose, we may assume that $F$ is trivial, so that $G\cong R\times S$. Now the maximal torus $\widetilde T$ of $R\times S$ is of the form $R\times T_S$, where $T_S$ is a maximal torus of $S$. Hence,
\[
 \frac{\vol(R\times S)}{\vol(\widetilde T)}= \frac{\vol(R\times S) }{\vol(R\times T_S)}=\frac{\vol(S)}{\vol(T_S)}.
\]
This shows that we may as well assume that $G=S$, that is, $G$ is semisimple and simply connected.\label{par:cptlgpstrth2}

\section{Euler-Maclaurin Formula}

Let $f$ be a smooth function on the real line. The Euler-Maclaurin formula  relates a sum $\sum_{x=0}^nf(x)$ to the integral $\int^n_0f(x)\,dx$ (Euler~\cites{eulerm1,eulerm2}, Maclaurin~\cite{emaclaurin}); precisely stated, for $N\in\N$,
\begin{equation}
 \sum_{x=0}^nf(x)=\int^n_0 f (x)\,dx 
	+\sum^N_{q=1}(-1)^q\frac{B_{q}}{q!} \bigl(f^{(q-1)}(n)-f^{(q-1)}(0)\bigr) +R_N,\label{eq:eulermac}
\end{equation}
where the coefficients $B_{q}$ are the Bernoulli numbers defined by the power series
\[
 \Td(x):= \frac{x}{1-e^{-x}}=\sum^\infty_{q=0}(-1)^q\frac{B_{q}}{q!}x^q, 
\]
and  $R_N$ is the remainder term, which can be estimated by
\begin{equation}
 |R_N|\le \frac{2\zeta(N)}{(2\pi)^{N}} \int^n_0\bigl|f^{(N)}(x)\bigr|\,dx.\label{eq:eulermaclaruinrem}
\end{equation}
Here $\zeta$ is the Riemann zeta function. Suppose $\sum^{\infty}_{x=0}f(x)$ exists and $\lim_{x\rightarrow\infty}f^{(q)}(x)=0$ for all $q\in\N$. Then we may set $n=\infty$ in Equation~\eqref{eq:eulermac}, which yields:
\begin{equation}
 \sum_{x=0}^\infty f(x)=\int^\infty_0 f (x)\,dx 
	+\sum^N_{q=1}(-1)^q\frac{B_{q}}{q!} f^{(q-1)}(0) +R_N.\label{eq:eulermacinfty}
\end{equation}
If, furthermore, $R_N\rightarrow0$ as $N\rightarrow\infty$ (for instance, when $f$ is a polynomial), we have:
\[
\sum_{x=0}^\infty f(x) = \Bigl.\Td\Bigl(\frac{\partial}{\partial h} \Bigl)\Bigr|_{h=0}\int^\infty_{-h}f(x)\,dx. 
\]
This expression for the Euler-Maclaurin formula first appeared in \mbox{Pukhlikov} and \mbox{Khovanski}~\cite{tdeulmac}.

The following lemma can be proved using the Euler-Maclaurin formula. We shall make use of it later on.
\begin{lem}\label{lem:eulmacprlm} Let $A$ and $B$ be real numbers with $A>0$, and let $m$ be any nonnegative integer. Let 
\[
 f_t(x) := x^{2m}e^{-t(Ax^2+Bx)}.  
\]
Consider the sum
\[
 S(t) :=\sum_{x=0}^\infty f_t(x)
\]
for $t>0$. Then,
\[
 t^{m+1/2}S(t)= t^{m+1/2}\int^\infty_0 f_t(x)\,dx + O(t^{1/2}) 
\]
for $t\rightarrow0+$.
\end{lem}
\begin{proof}
By Equation~\eqref{eq:eulermacinfty} with $N=1$,
\[
 S(t)=\int^\infty_0 f_t (x)\,dx 
	- B_{1} f_t(0) +R_1.
\]
By the estimate~\eqref{eq:eulermaclaruinrem}, the absolute value of $R_1$ is bounded by, up to a constant factor,
\[
   \int^\infty_0 \Bigl|
	2mx^{2m-1}e^{-t(Ax^2+Bx)} -t (2Ax^{2m+1}+Bx^{2m})e^{-t(Ax^2+Bx)}\Bigr|\,dx.
\]
By a change of variable of the type $x\mapsto \alpha x+\beta$, $\alpha>0$, the above integral can be recast in the following form:
\[
 \int^\infty_\alpha \Bigl|P_1(x)e^{-tx^2} +tP_2(x)e^{-tx^2}\Bigr|\, dx,
\]
where $P_1(x)$ and $P_2(x)$ are polynomials of degree $2m-1$ and $2m+1$, respectively. This integral is bounded by, up to a constant factor,
\begin{equation}
\int^\infty_{-\infty} |P_1(x)|e^{-tx^2}dx +t\int^\infty_{-\infty}|P_2(x)|e^{-tx^2}\, dx.
\label{eq:eulmaclem}
\end{equation}
It is known that, for a nonnegative integer $n$,
\[
\int^\infty_{-\infty} |x^n|e^{-tx^2}dx = \frac{1}{t^{(n+1)/2}}\Gamma\Bigl(\frac{n+1}{2}\Bigr),
 \]
where $\Gamma$ denotes the Gamma function. Hence, we see that the integral~\eqref{eq:eulmaclem} is bounded by a quantity that is of $O(t^{-m})$ for $t\rightarrow0+$. Meanwhile, the integral $\int_{0}^\infty f_t(x)\,dx$ is of $O(t^{-m-1/2})$. Hence,
\[
 t^{m+1/2}S(t) = t^{m+1/2}\int^{\infty}_0f_t(x)\,dx + O(t^{1/2}).\qedhere
\]
\end{proof}

\section{Proof of Harish-Chandra's  Formula} \label{par:introgencase} 

\begin{notation}
We use the following notation for the dimensions of $G$, $T$, and $G/T$:
\[
 n:=\dim G,\quad r:=\dim T,\quad 2m:=\dim (G/T)=n-r.
\]
We denote by $\Phi^+$ the selected set of positive roots of $\mathfrak{g}$, and set $\rho=\frac{1}{2}\sum_{\alpha\in\Phi^+}\alpha$. The selection of $\Phi^+$ amounts to choosing a Weyl chamber in $\mathfrak{t}^*$, which we designate by $K$ and refer to as the fundamental Weyl chamber. The fundamental weights are denoted by $\{\lambda_i\}_{i=1}^r$. The lattice of (analytically) integral weights is represented by $\Lambda$. For each $\lambda\in\Lambda\cap K$, we denote by $V_\lambda$ the irreducible $G$-vector space with highest weight $\lambda$.  (For reference on the terminologies, see \mbox{Duistermaat} and \mbox{Kolk}~\cite{duistermaatkolk}*{Chs. 3--4}.)
\end{notation}
\begin{rmk}
Let $\mathfrak{p}$ be the orthogonal complement of $\mathfrak{t}$ in $\mathfrak{g}$. The complexification $\mathfrak{p}_\C$ of $\mathfrak{p}$ is the direct sum of root spaces of $\mathfrak{g}$. Since half of the roots are positive, we have
\begin{equation}
m = |\Phi^+|.
\label{eq:dimgmtposrt}
\end{equation}
\end{rmk}

\begin{prop}\label{prop:cptliehttrpr}
For the heat trace $Z(t)$ of the Laplacian on $G$, we have:
\[ Z(t) =\sum_{\lambda\in \Lambda\cap K}\dim(V_\lambda)^2 e^{tC_\lambda},\]
where
\[ C_\lambda=-\|\lambda+\rho\|^2+\|\rho\|^2.\]
(Here $\|\cdot \|$ is the norm on $\mathfrak{t}^*$ induced by the inner product on $\mathfrak{g}$.)
\end{prop}
\begin{proof}
Owing to the representation theory of compact Lie groups, there is a one-to-one correspondence between the unitary dual $\hat G$ and $\Lambda\cap K$. Now the Peter-Weyl theorem states that we have a $(G\times G)$-equivariant Hilbert space isomorphism (\mbox{Peter} and \mbox{Weyl}~\cite{peterweyl}):
\begin{equation}
L^2(G,\C) \cong  \bigoplus_{\lambda\in\Lambda\cap K} {V_\lambda^*\otimes V_\lambda}. 
\label{eq:peterweyl}
\end{equation}
Here $L^2(G,\C)$ denotes the space of square-integrable complex-valued functions on $G$. As a consequence of the invariance of the metric on $G$, the action induced on the right-hand side of~\eqref{eq:peterweyl} by the Laplacian is that of $\mathbf{1}\otimes \Omega$, where $\Omega$ denotes the Casimir element in the universal enveloping algebra of $\mathfrak{g}$. By Schur's lemma,  the action of $\Omega$ on $V_\lambda$ is by a constant, say $C_\lambda$, whose value is well-known to be as asserted in the proposition (see, for instance, Knapp~\cite{knapp}*{p.~295}).
\end{proof}

\begin{prop}\label{prop:modihtr2}
We have:
\[ Z(t) = e^{t\|\rho\|^2}\sum_{\lambda\in \Lambda\cap K} d(\lambda)^2 e^{-t\|\lambda\|^2}, 
\]
where $d$ is a function on $\mathfrak{t}^*$ defined by
\[ d(\lambda) := \frac{\prod_{\alpha\in\Phi^+}\langle \alpha,\lambda\rangle}{\prod_{\alpha\in\Phi^+}\langle \alpha,\rho\rangle}.\]
(Here we are using, with a slight abuse of notation, $\langle \cdot , \cdot \rangle$ to denote the inner product on $\mathfrak{t}^*$ induced by that on $\mathfrak{g}$.)
\end{prop}
\begin{rmk}
The argument presented below is from Fegan~\cite{fegan}*{p.~594}:
\end{rmk}
\begin{proof}
In terms of the function $d$, we may express the Weyl dimension formula as:
\[
 \dim(V_\lambda) = d(\lambda+\rho).
\]
Hence, by Proposition~\ref{prop:cptliehttrpr}, we have
\[
 Z(t)  = e^{t\|\rho\|^2}\sum_{\lambda\in \Lambda\cap K} d(\lambda+\rho)^2 e^{-t\|\lambda+\rho\|^2}.
\]
Note that
\[
\sum _{\lambda\in \Lambda\cap K} d(\lambda+\rho)^2 e^{-t\|\lambda+\rho\|^2}
	=\sum _{\lambda\in (\Lambda\cap K)+\rho} d(\lambda)^2 e^{-t\|\lambda\|^2}.
\]
The shifted index set $(\Lambda\cap K)+\rho$ is the set of weights that lie in the interior $K^\circ$ of the fundamental Weyl chamber. Hence,
\begin{equation}
 Z(t) = e^{t\|\rho\|^2}\sum_{\lambda\in \Lambda\cap K^\circ} d(\lambda)^2 e^{-t\|\lambda\|^2}.
\label{eq:cptlghtrmd}
\end{equation}
But since the boundary $\partial K$ of the fundamental Weyl chamber is contained in the hyperplanes orthogonal to the roots, the restriction of $d$ to $\partial K$ is zero. Hence we may replace the index set $\Lambda\cap K^\circ$ of the sum~\eqref{eq:cptlghtrmd} to $\Lambda\cap K$. This proves the proposition.
\end{proof}

\begin{lem}\label{lem:htrldt}
Let $\mu_{\grt^*}$ denote the Lebesgue measure on $\grt^*$ induced by the inner product $\langle \cdot , \cdot \rangle$. Let $P$ be the fundamental parallelepiped in $\mathfrak{t}^*$ formed by the fundamental weights of $G$.
\begin{enumerate}
\item \label{item:htrldt} For $t\rightarrow0+$, we have
\begin{equation}
 t^{n/2} Z(t) = t^{n/2} I(t) +O(t), \label{eq:httremaprx}
\end{equation}
where
\begin{equation}
 I(t)= \int_{K} d(\lambda)^2e^{-t\|\lambda\|^d}\,\frac{\mu_{\grt^*}(\lambda)}{\vol(P)}. \label{eq:httrapprxint}
\end{equation}

\item \label{item:torusvol} The volume of the maximal torus is related to $\vol(P)$ by
\begin{equation}
\vol(T)= \frac{(2\pi)^r}{\vol(P)}. \label{eq:torusvol}
\end{equation}
\end{enumerate}
\end{lem}
\begin{proof}
(\ref{item:htrldt}) By Proposition~\ref{prop:modihtr2} we have:
\[ t^{n/2}Z(t) =t^{n/2}\bar Z(t) + O(t),\]
where
\[ \bar Z(t) := \sum_{\lambda\in \Lambda\cap K} d(\lambda)^2 e^{-t\|\lambda\|^2}.\]
Let $(x_1,\dotsc,x_r)$ denote the component variables on $\grt^*$ relative to the fundamental weights $\{\lambda_i\}_{i=1}^r$, so that an arbitrary element $\lambda\in\mathfrak{t}^*$ is expressed as $\lambda=\sum_{i=1}^rx_i\lambda_i$. Then $\bar Z(t)$ is of the form:
\[
 \bar Z(t) =\sum_{x_1=0}^\infty\dotsb\sum_{x_r=0}^\infty d(x_1,\dotsc,x_r)^2 e^{-tq(x_1,\dotsc,x_r)},
\]
where $d$ and $q$ are homogeneous polynomials of degree $m$ and $2$, respectively. Applying Lemma~\ref{lem:eulmacprlm} iteratively to $\bar Z(t)$, we have:
\[ t^{n/2}\bar Z(t)  = t^{n/2}I(t) +O(t),\]
where 
\[
 I(t) = \int^\infty_{0}\dotsi\int^\infty_{0} d(x_1,\dotsc,x_r)^2 e^{-tq(x_1,\dotsc,x_r)}\,dx_1\dotsm dx_r.
\]
This is the integral~\eqref{eq:httrapprxint}.

(\ref{item:torusvol}) Recall our assumption that $G$ is compact, connected, simply connected, and semisimple. In this case the fundamental weights $\{\lambda_i\}_{i=1}^r$ form a basis for $\Lambda$. Let $\{ H_i \}_{i=1}^r$ be the simple coroots, that is, the vectors in $\mathfrak{t}$ that are dual to the fundamental weights relative to the inner product. Let $\hat \Lambda$ be the $\Z$-lattice spanned by the simple coroots. Then
\[ 2\pi\hat\Lambda:= \exp^{-1}\{e\}\cap\mathfrak{t}, \]
where $\exp:\mathfrak{g}\to G$ is the exponential map, and $e$ is the identity in $G$. 

Let $Q$ be the fundamental parallelepiped in $\grt$ formed by the simple coroots. Let $\mu_\grt$ be the Lebesgue measure on $\grt$ induced by the inner product. Since $\exp(2\pi H_i)=e$ for all simple coroots $H_i$,  we have
\[
 \vol(T) = (2\pi)^r\vol(Q).
\]
Meanwhile, because the lattice $\hat\Lambda$ and $\Lambda$ are dual to each other, we have
\[ \vol(P) = \frac{1}{\vol(Q)}.\]
Thus we have Equation~\eqref{eq:torusvol}.
\end{proof}

\begin{thm}
Let $G$ be a compact connected Lie group and $T$ its maximal torus. Let $\grg$ and $\grt$ be their Lie algebras, respectively. Let $\langle \cdot , \cdot \rangle$ denote an $\Ad(G)$-invariant inner product on $\grg$ and also the induced inner product on the dual space $\grg^*$. Endow $G$ and $T$ with the measures induced by the bi-invariant metric that is generated by $\langle \cdot , \cdot \rangle$. Then, with the quotient measure on $G/T$, we have:
\[
 \vol(G/T) =\prod_{\alpha\in\Phi^+}\frac{2\pi}{\langle \alpha,\rho\rangle},\]
where $\Phi^+$ is the set of the selected positive roots of $G$ and $\rho=\frac{1}{2}\sum\limits_{\alpha\in\Phi^+}\alpha$.
\end{thm}
\begin{proof}
As explained in Section~\ref{par:cptlgpstrth1}, we may assume that $G$ is simply connected and semisimple. Let $W$ denote the Weyl group of $G$. The action of $W$ on $\mathfrak{t}^*$ preserves the set of roots and the inner product. This implies that  $d(\lambda)^2$ and $\|\lambda\|$, which appear in the integral~\eqref{eq:httrapprxint}, are preserved under the $W$-action; so it is possible to extend the domain of integration to all of $\grt^*$ as follows:
\[
I(t) = \frac{1}{|W|} \int_{\grt^*} d(\lambda)^2 e^{-t\|\lambda\|^2}\frac{\mu_{\grt^*}(\lambda)}{\vol(P)}.\]
Then, by Equation~\eqref{eq:torusvol},
\[
I(t) = \frac{1}{|W|} \frac{\vol(T)}{(2\pi)^r} \int_{\grt^*} d(\lambda)^2 e^{-t\|\lambda\|^2}\mu_{\grt^*}(\lambda).
\]

We now switch the domain of integration from $\grt^*$ to $\grt$ via the linear isomorphism
\[
\begin{array}{ccc}
 \grt^*&\rightarrow&\grt,\\
	\lambda&\mapsto&X_\lambda,
\end{array}
\]
where $X_\lambda$ is the vector that is dual to $\lambda$ relative to the inner product. This isomorphism is precisely through which the inner product on $\grt$ was transferred to $\grt^*$; in particular, the Jacobian determinant of this isomorphism is $1$. Moroever, we have: 
\[
 d(\lambda)^2= \frac{\prod_{\alpha\in\Phi^+}\langle \alpha,\lambda\rangle^2}{\prod_{\alpha\in\Phi^+}\langle \alpha,\rho\rangle^2} = \frac{\prod_{\alpha\in\Phi^+}\alpha(X_\lambda)^2}{\prod_{\alpha\in\Phi^+}\langle \alpha,\rho\rangle^2},
\]
and
\[
 e^{-t\|\lambda\|^2}=e^{-t\|X_\lambda\|^2}.
\]
Hence, 
\begin{equation}
 I(t) =\frac{1}{|W|} \frac{\vol(T)}{(2\pi)^r}\Bigl(\prod_{\alpha\in\Phi^+}\frac{1}{\langle \alpha,\rho\rangle^2}\Bigr) \int_{\grt} \Bigl(\prod_{\alpha\in\Phi^+}\alpha(X)^2\Bigr) e^{-t\|X\|^2}\,\mu_{\grt}(X), \label{eq:volfmlder1} 
\end{equation}
where $\mu_\grt$ is the Lebesgue measure on $\grt$ induced by the inner product.

Recall the Weyl integration formula, which states that 
\begin{equation}
 \int_{\grg} f(X)\, \mu_\grg(X) = \frac{1}{|W|}\frac{\vol(G)}{\vol(T)}\int_\grt f(X)\, \Bigl(\prod_{ \alpha\in\Phi^+}\alpha(X)^2\Bigr)\,\mu_\grt(X) \label{eq:wifexplct}
\end{equation}
for any function $f$ that is invariant relative to the adjoint action of $\grg$ on itself (see \mbox{Duistermaat} and \mbox{Kolk}~\cite{duistermaatkolk}*{Thm.~3.14.1, p.~185}). Here the measure $\mu_\grg$ on $\grg$ is again the Lebesgue measure  induced by the inner product. Substitution of $f(X)=e^{-t\|X\|^2}$ in Equation~\eqref{eq:wifexplct} and some rearranging of terms leads us to:
\[
 \int_\grt e^{-t\|X\|^2} \Bigl(\prod_{ \alpha\in\Phi^+}\alpha(X)^2\Bigr)\,\mu_\grt(X) = |W|\frac{\vol(T)}{\vol(G)} \int_{\grg} f(X)\, \mu_\grg(X).
\]
Implementing this on Equation~\eqref{eq:volfmlder1}, we get
\begin{equation}
 I(t) =  \frac{\vol(T)^2}{(2\pi)^r\vol(G)}\Bigl(\prod_{\alpha\in\Phi^+}\frac{1}{\langle \alpha,\rho\rangle^2}\Bigr)
	\int_{\grg}  e^{-t\|X\|^2} \,\mu_\grg(X). \label{eq:sslvolform}
\end{equation}
The last integral is just a Gaussian integral, which has the value:
\[
 \int_{\grg}  e^{-t\|X\|^2} \,\mu_\grg(X) = \Bigl( \frac{\pi}{t} \Bigr)^{n/2}.
\]
Therefore,
\begin{equation}
 I(t)= \frac{\vol(T)^2}{(2\pi)^r\vol(G)}\Bigl(\prod_{\alpha\in\Phi^+}\frac{1}{\langle \alpha,\rho\rangle^2}\Bigr)\Bigl( \frac{\pi}{t} \Bigr)^{n/2}.\label{eq:heattrint}
\end{equation}
Inserting this expression into Equation~\eqref{eq:httremaprx} and invoking Weyl's law (Equation~\eqref{eq:weylslaw}), we get
\[
 \frac{\vol(G)^2}{\vol(T)^2}=(2\pi)^{2m}\prod_{\alpha\in\Phi^+}\frac{1}{\langle \alpha,\rho\rangle^2}. 
\]
The inner product between two positive roots is nonnegative. Hence, 
\[
 \frac{\vol(G)}{\vol(T)}=(2\pi)^{m}\prod_{\alpha\in\Phi^+}\frac{1}{\langle \alpha,\rho\rangle}. 
\]
By Equation~\eqref{eq:dimgmtposrt}, we have
\[
 \frac{\vol(G)}{\vol(T)}=\prod_{\alpha\in\Phi^+}\frac{2\pi}{\langle \alpha,\rho\rangle}.
\]
This proves the formula.
\end{proof}

\section{Concluding Remarks}

As we have mentioned in the introduction, one significance of the present argument is that, having equipped with some key results of Lie theory and differential geometry, namely, the representation theory of compact Lie groups and Weyl's law, one can easily access Harish-Chandra's formula by an elementary means of the Euler-Maclaurin formula.

What seems to be of greater significance, however, lies in the question that it stirs up when one brings this approach to the Atiyah-Singer index theorem. Suppose $D$ is a Dirac operator on a closed spin manifold $M$. Let $\Ind(D)$ be the (graded) index of $D$. In its simplest setting, the Atiyah-Singer index theorem states that 
\[
 \Ind(D) =\int_{M}\hat A, \label{eq:atiyahsinger}
\]
where $\hat A$ is the Hirzebruch $\hat A$-class of $M$ (Atiyah and Singer~\cite{atiyahsinger}). Meanwhile, it is also true that
\[
\Ind(D) = \Str(e^{tD^2}), 
\]
where $\Str$ denotes the super trace (trace over the even domain minus the trace over the odd domain); this is known as the McKean-Singer formula (McKean and Singer~\cite{mckeansinger}*{p.~61}). Therefore,
\begin{equation}
  \Str(e^{tD^2})=\int_M\hat A.\label{eq:atsimc}
\end{equation}
At this point we may apply on the left-hand side the Euler-Maclaurin formula, which involves the power series
\[ \Td(x) = \frac{x}{1-e^{-x}}.\]
But this power series is what produces the $\hat A$-class, which is on the right-hand side of Equation~\eqref{eq:atsimc},  under the Chern-Weil homomorphism for the tangent bundle of $M$. It is a mystery (at least for me) whether this is just a fluke; or can the Euler-Maclaurin formula explain the appearance of the $\hat A$-class in the index formula? The author welcomes suggestions towards answering this question.

\end{document}